\documentclass[a4paper,11pt,reqno]{customart}

\usepackage[utf8x]{inputenc}
\usepackage[margin=1in]{geometry}
\usepackage{verbatim}
\usepackage{color}
\usepackage[table]{xcolor}
\usepackage{listings}
\usepackage{amssymb,amsfonts,amsthm,amsmath}
\usepackage{url}
\usepackage{enumerate}
\usepackage{todonotes}
\usepackage[all]{xy}
\usepackage{multirow}
\usepackage{tikz}
\usepackage[underline=false]{pgf-umlsd}
\usepackage{stmaryrd}
\usepackage{footnote}
\makesavenoteenv{tabular}
\usepackage{hyperref}

\makeatletter
\newcommand\footnoteref[1]{\protected@xdef\@thefnmark{\ref{#1}}\@footnotemark}
\makeatother

\hypersetup{
	pdfborder={0 0 0}
}

\definecolor{grey}{rgb}{0.95,0.95,0.95}
\definecolor{green}{rgb}{0.2,0.6,0.4}

\newcommand{\N}{\mathbb{N}}

\renewcommand{\O}{\textbf{0}}

\newcommand{\conc}{{{}^\smallfrown}}

\newcommand{\inter}{\mbox{int}}

\newcommand{\Psf}{\mathsf{P}}
\newcommand{\Qsf}{\mathsf{Q}}

\newcommand{\imp}{\rightarrow}

\newcommand{\Nb}{\mathbb{N}}

\newcommand{\Qb}{\mathbb{Q}}

\newcommand{\Ccal}{\mathcal{C}}

\newcommand{\Ncal}{\mathcal{N}}

\newcommand{\Qcal}{\mathcal{Q}}
\newcommand{\Rcal}{\mathcal{R}}
\newcommand{\Scal}{\mathcal{S}}

\newcommand{\uh}{{\upharpoonright}}

\renewcommand{\setminus}{\smallsetminus}

\newcommand{\tuple}[1]{\left\langle #1 \right\rangle}

\DeclareMathAlphabet{\mathbfsf}{\encodingdefault}{\sfdefault}{bx}{sl}

\newcommand{\s}[1]{\ensuremath{\sf{#1}}}

\newcommand{\is}[1]{\s{I}\Sigma^0_{#1}}
\DeclareMathOperator{\erp}{\s{(\aleph_0,\eta)}^2}
\DeclareMathOperator{\ers}{\s{(\eta)^1_{<\infty}}}
\DeclareMathOperator{\erps}{\s{(\aleph_0,\eta)^1}}
\DeclareMathOperator{\rca}{\s{RCA}_0}
\DeclareMathOperator{\aca}{\s{ACA}_0}

\DeclareMathOperator{\bst}{\s{B}\Sigma^0_2}
\DeclareMathOperator{\ist}{\s{I}\Sigma^0_2}

\DeclareMathOperator{\er}{\s{ER}}
\DeclareMathOperator{\rt}{\s{RT}}

\DeclareMathOperator{\tto}{\s{TT}}

\DeclareMathOperator{\Red}{\textup{\texttt{red}}}
\DeclareMathOperator{\blue}{\textup{\texttt{blue}}}

\usetikzlibrary{shapes,arrows}
\usetikzlibrary{decorations.markings}
\definecolor{lightblue}{HTML}{e6e6e6}
\definecolor{lightred}{HTML}{eca6a6}
\definecolor{lightgreen}{RGB}{164,244,140}

\newtheoremstyle{custom}% name of the style to be used
  {10pt}% measure of space to leave above the theorem. E.g.: 3pt
  {10pt}% measure of space to leave below the theorem. E.g.: 3pt
  {\normalfont}% name of font to use in the body of the theorem
  {}% measure of space to indent
  {\bfseries}% name of head font
  {}% punctuation between head and body
  { }% space after theorem head; " " = normal interword space
  {}% Manually specify

\theoremstyle{custom}

\usepackage{xcolor}	
\usepackage{soul}

\newtheorem{theorem}{Theorem}[section]
\newtheorem{lemma}[theorem]{Lemma}
\newtheorem{definition}[theorem]{Definition}
\newtheorem{remark}[theorem]{Remark}

\newtheorem{question}[theorem]{Question}

\begin{document}

\title[Coloring the rationals in reverse mathematics]
	{Coloring the rationals in reverse mathematics}
%\author{
%  Emanuele Frittaion \and Ludovic Patey
%}

\author[Frittaion]{Emanuele Frittaion}
 \address{Mathematical Institute, Tohoku University, Japan}
\thanks{Emanuele Frittaion's research is supported by the Japan Society for the
Promotion of Science.}
\email{frittaion@math.tohoku.ac.jp}
\urladdr{http://www.math.tohoku.ac.jp/~frittaion/}

\author[Patey]{Ludovic Patey}
 \address{Laboratoire PPS, Paris Diderot University, France}
\thanks{Ludovic Patey is funded by the John Templeton Foundation (`Structure and 
Randomness in the Theory of Computation' project). The opinions expressed in this 
publication are those of the author(s) and do not necessarily reflect the views of the 
John Templeton Foundation.}
\email{ludovic.patey@computability.fr}
\urladdr{http://ludovicpatey.com/}

\begin{abstract}
Ramsey's theorem for pairs asserts that every $2$-coloring of the
pairs of integers has an infinite monochromatic subset. 
In this paper, we study a strengthening of Ramsey's theorem for pairs
due to Erd\H{o}s and Rado, which states that every
$2$-coloring of the pairs of rationals has either an infinite
0-homogeneous set or a 1-homogeneous set of order type $\eta$,
where~$\eta$ is the order type of the rationals.
This theorem is a natural candidate to lie strictly between the arithmetic comprehension axiom
and Ramsey's theorem for pairs.
This Erd\H{o}s-Rado theorem, like the tree theorem for pairs, 
belongs to a family of Ramsey-type statements whose logical strength
remains a challenge.
\end{abstract}

\maketitle

\section{Introduction}

In this paper, we investigate the reverse mathematics of a well-known theorem due to 
Erd\H{o}s and Rado about $2$-colorings of pairs of rationals. This theorem is a natural strengthening
of Ramsey's theorem for pairs and two colors. We say that an order type $\alpha$ is \emph{Ramsey}, and write~$\alpha \to (\alpha)^2_2$, if for every coloring $f\colon[L]^2\to 
2$, where $L$ is a linear order of order type $\alpha$, there is a homogeneous set $H$ 
such that $(H,\leq_L)$ has order type $\alpha$. Ramsey's theorem for pairs and two colors asserts that $\omega$ is Ramsey.
It turns out that~$\omega$  and~$\omega^{*}$ are the only countable Ramsey order types.  In particular, $\eta \to (\eta)^2_2$ does not hold, where $\eta$ is the order type of the 
rationals. A standard counterexample is as follows. Fix a 
one-to-one map 
$i\colon\Qb\to\Nb$. Define $f\colon[\Qb]^2\to 2$ by 
letting
\[   f(x,y)=\begin{cases}
             0   &   \text{if } x<_\Qb y \land i(x)<i(y); \\
             1 &     \text{if } x<_\Qb y\land i(x)>i(y).
            \end{cases} \]
A  homogeneous set of order type $\eta$ would give an embedding of $\Qb$ into $\omega$ (with color $0$) or 
$\omega^*$ (with color $1$), which is impossible. Even though Ramsey's theorem for 
rationals fails,  Erd\H{o}s and Rado \cite[Theorem 
4, p. 427]{Erdos1952Combinatorial} proved the following
Ramsey-type theorem (see also Rosenstein \cite[Theorem 11.7, p. 
207]{Rosenstein1982Linear}).

\begin{theorem}[Erd\H{o}s-Rado theorem]\label{ER52}
The partition relation $\eta\to(\aleph_0,\eta)^2$ holds.
\end{theorem}
The relation $\eta\to(\aleph_0,\eta)^2$ asserts that for every coloring $f\colon[L]^2 
\to 2$, where $L$ is a linear order of order type~$\eta$, there is either an infinite  
0-homogeneous set or a 1-homogeneous set $H$ such that $(H,\leq_\Qb)$ has  order type~$\eta$.\smallskip 

We study Theorem \ref{ER52} within the framework of reverse mathematics (see Simpson 
\cite{Simpson2009Subsystems}). Reverse mathematics is a vast mathematical program whose 
goal is to study the logical strength of ordinary theorems in terms of set existence 
axioms. It uses the framework of subsystems of second-order arithmetic,
with the base theory $\rca$ (recursive comprehension axiom).
$\rca$ is composed of $P^-$, that is, the basic first-order Peano axioms for 
$0,1,+,\times,<$, together with $\Delta^0_1$-comprehension and $\Sigma^0_1$-induction 
with number and set parameters.
$\rca$ is usually thought of as capturing \emph{computable mathematics}.
It turns out that the large majority of countable mathematics can be
proven in~$\aca$, where~$\aca$ is $\rca$ together with arithmetic comprehension.
See Hirschfeldt~\cite{Hirschfeldt2014Slicing} for a gentle presentation of the reverse 
mathematics below~$\aca$.\smallskip

We formalize Theorem \ref{ER52} in~$\rca$ as follows.
\begin{itemize}
 \item[$\erp$] For every coloring $f\colon[\Qb]^2\to 2$ there exists either an infinite $0$-homogeneous set or a $1$-homogeneous set $H$ such that $(H,\leq_\Qb)$ is dense.
\end{itemize}
Here $\Qb$ is any fixed primitive recursive presentation of the rationals. We may safely assume 
that the domain of $\Qb$ is $\Nb$.  Note that provably in $\rca$ every two (countable) linear orders of order type $\eta$ are isomorphic and any dense linear order obvioulsy contains a linear order of order type $\eta$. Therefore  $\erp$ is provably equivalent over $\rca$ to the statement of Theorem \ref{ER52}.

In order to study  $\erp$ we also consider a 
version of the infinite pigeonhole principle over the rationals, namely the statement:
\begin{itemize}
 \item[$\ers$] For every $n$ and for 
every  $n$-coloring $f\colon\Qb\to n$ there exists a dense homogeneous set.
\end{itemize}
\smallskip

The early study of reverse mathematics has led to 
the observation
that most of the theorems happen to be equivalent to five
main subsystems of second-order arithmetic that 
Montalb\'{a}n~\cite{Montalban2011Open} called the ``Big Five''.
However, Ramsey's theory provides many statements escaping  
this observation.
Perhaps the most well-known example is Ramsey's theorem for pairs and two colors ($\rt^2_2$). The 
effective analysis of Ramsey's theorem was started by 
Jockusch~\cite{Jockusch1972Ramseys}. In the framework of reverse mathematics, Simpson 
(see \cite{Simpson2009Subsystems}), building on Jockusch results, proved that 
whenever~$n \geq 3$ and $k\geq 2$, $\rt^n_k$ is equivalent to~$\aca$ 
over~$\rca$. The case of~$\rt^2_2$ had been a long-standing open problem
until Seetapun~\cite{Seetapun1995strength} proved that~$\rt^2_2$ is strictly weaker 
than~$\aca$ over~$\rca$. Cholak, Jockusch and Slaman \cite{Cholak2001strength} paved the 
way to the reverse mathematics analysis of Ramsey's theorem for pairs.  Since then,
many consequences of Ramsey's theorem for pairs have been studied,
leading to a whole zoo of independent statements. However, no natural statement besides Ramsey's theorem for pairs ($\rt^2$) is known to be strictly between $\aca$ and~$\rt^2_2$ over~$\rca$. The only known candidate is the 
tree theorem for pairs ($\tto^2_2$) studied in~\cite{Chubb2009Reverse,Corduan2010Reverse, 
Dzhafarov2010Ramseys, Patey2015strength}. We show that  $\erp$ also lies between $\aca$ 
and $\rt^2_2$, and so represents another candidate, arguably more natural than 
$\tto^2_2$.

Although no relation is known between them,
$\tto^2_2$ and $\erp$ share some essential
combinatorial features and put the emphasis on a new family of Ramsey-type theorems,
characterized by what we call a \emph{disjoint extension commitment}. See section~\ref{sect:questions}
for a discussion on this notion. Some separations known for variants of $\tto^2_2$
are essentially due to this common feature, which enables us to prove the same separations
for variants of $\erp$. In particular,
we prove that $\erp$  does not computably reduce to Ramsey's theorem 
for pairs \emph{with an arbitrary number of colors} ($\rt^2$). However, we 
cannot simply adapt this ``one-step separation'' to a separation over $\omega$-models, 
and in particular over~$\rca$, as in the case of  $\tto^2_2$ 
\cite{Patey2015strength}. This is the first known example of such an inability.
Indeed, a diagonalization against an $\rt^2_4$-instance is similar
to a diagonalization against two $\rt^2_2$-instances. Therefore, diagonalizing against~$\rt^2$
has some common flavor with a separation over standard models. 

Among the consequences of Ramsey's theorem for pairs, Ramsey's theorem for 
singletons ($\rt^1$), also known as the infinite pigeonhole principle, is of 
particular interest. $\rt^1$
 happens to be equivalent to the $\Sigma^0_2$ bounding
scheme (see Hirst~\cite{Hirst1987Combinatorics}).
The $\Sigma^0_2$ bounding scheme ($\bst$) is formally defined as
\[
(\forall x<a)\exists y\varphi(x,y,a)\implies \exists b(\forall 
x<a)(\exists y<b)\varphi(x,y,n)
\]
where $\varphi$ is any $\Sigma^0_2$ formula. 
One may think of $\bst$ as asserting that the finite
union of finite sets is finite (see for instance \cite{FriMar12}). We 
show that $\ers$, the corresponding pigeonhole principle for rationals, is 
strictly stronger than $\bst$, and hence has the same reverse mathematics 
status as the tree theorem for singletons ($\tto^1$) \cite{Corduan2010Reverse}.
\smallskip

For the purpose of separating  $\erp$ from 
$\rt^2$ over computable reducibility, we also 
introduce the asymmetric version of $\ers$ for two colors, namely $\erps$, 
stating that for every partition $A_0\cup A_1=\Qb$ of the rationals there exists either 
an infinite subset of $A_0$ or a dense subset of $A_1$. Indeed, we show the existence 
of a $\Delta^0_2$-instance of $\erps$, and hence of a computable instance of 
$\erp$, which does not reduce to any computable instance of $\rt^2$.

\subsection{Definitions and notation}

\emph{String}.
A \emph{string} is an ordered tuple of bits $b_0, \dots, b_{n-1}$, that 
is, such that~$b_i < 2$ for every~$i < n$.
The empty string is written $\langle\rangle$. A \emph{real}
is an infinite listing of bits $b_0, b_1, \dots$.
Given $s \in \omega$, $2^s$ is the set of strings of length $s$ and
$2^{<s}$ is the set of strings of length $<s$. Similarly,
$2^{<\omega}$ is the set of finite strings and $2^{\omega}$ is the set of reals.
Given a string $\sigma \in 2^{<\omega}$, we denote by $|\sigma|$ its length.
Given two strings $\sigma, \tau \in 2^{<\omega}$, we write $\sigma\conc\tau$ for the 
concatenation of $\sigma$ and $\tau$, and we say that $\sigma$ is a \emph{prefix}
of $\tau$ (written $\sigma \preceq \tau$) if there exists a string $\rho \in 2^{<\omega}$
such that $\sigma\conc\rho = \tau$.  Given a real $X$, we write $\sigma \prec X$ if
$\sigma = X \uh n$ for some $n \in \omega$, where $X \uh n$ denotes the restriction of $X$ to its first $n$ elements.
We may identify a real with a set of integers by considering that the real is its characteristic function.

\emph{Tree, path}.
A \emph{binary tree} $T \subseteq 2^{<\omega}$ is a set downward-closed under the prefix relation.
A real $P$ is a \emph{path} though~$T$ if for every $\sigma \prec P$, $\sigma \in T$.

\emph{Sets, partitions}.
Given two sets $A$ and $B$, we denote by $A < B$ the formula
$(\forall x \in A)(\forall y \in B)[x < y]$
and by $A \subseteq^{*} B$ the formula $(\forall^{\infty} x \in A)[x \in B]$,
meaning that $A$ is included in $B$ \emph{up to finitely many elements}.
Given a set~$X$ and some integer~$k$, a~\emph{$k$-partition of~$X$}
is a $k$-uple of pairwise disjoint sets $A_0, \dots, A_{k-1}$ such that~$A_0 \cup \dots \cup A_{k-1} = X$.
A \emph{Mathias condition} is a pair $(F, X)$
where $F$ is a finite set, $X$ is an infinite set
and $F < X$.
A condition $(F_1, X_1)$ \emph{extends } $(F, X)$ (written $(F_1, X_1) \leq (F, X)$)
if $F \subseteq F_1$, $X_1 \subseteq X$ and $F_1 \setminus F \subset X$.
A set $G$ \emph{satisfies} a Mathias condition $(F, X)$
if $F \subset G$ and $G \setminus F \subseteq X$.

\section{The Erd\H{o}s-Rado theorem in reverse mathematics}\label{sect:partitions-reducibility}

We start off the analysis of the Erd\H{o}s-Rado theorem by proving that the 
statement $\erp$ lies between $\aca$ and~$\rt^2_2$. On the lower bound hand,
$\erp$ can be seen as an immediate strengthening
of $\rt^2_2$. The upper bound is an effectivization of the original proof
of~$\erp$ by Erd\H{o}s and Rado in~\cite{Erdos1952Combinatorial}. 

\begin{lemma}[$\rca$]
$\erp\rightarrow\rt^2_2$. 
\end{lemma}
\begin{proof}
An instance of $\rt^2_2$ can be regarded as an instance of $\erp$. Moreover, 
provably in $\rca$, a dense set is infinite.
\end{proof}

The rest of this section is devoted to show that $\erp$ is provable in $\aca$. For this 
purpose, we give the following definition.

\begin{definition}[$\rca$]
By \emph{interval} we mean a set of the form $I=(x,y)_\Qb$ for $x,y\in \Qb$. We say that
$A\subseteq\Qb$ is \emph{somewhere dense} if $A$ is dense in some
interval of $\Qb$,  i.e.,\ there exists an interval $I$ such that for all intervals 
$J\subseteq I$ we have that $A\cap J\neq\emptyset$. We call $A$ \emph{nowhere 
dense}  otherwise.
\end{definition}
Notice that the above notion of nowhere dense is the usual topological notion with respect to
the order topology of $\Qb$. In general, the nowhere dense sets of a topological 
space form an ideal. This is crucial in the proof by Erd\H{o}s and Rado. For this reason, 
we also use the terminology \emph{positive} and \emph{small} for somewhere dense and 
nowhere dense respectively. In $\rca$ we can show that nowhere dense 
subsets of $\Qb$ are small, meaning that:
\begin{enumerate}[\quad $(1)$]
 \item If $A\subseteq\Qb$ is small and $B\subseteq A$, then $B$ is small;
 \item If $A,B\subseteq\Qb$ are small, then $A\cup B$ is small.
\end{enumerate}

With enough induction, it is possible to generalize $(2)$ to finitely 
many sets. 

\begin{lemma}[$\rca+\ist$]\label{small}
If $A_i$ is a small subset of $\Qb$ for all $i<n$, then $\bigcup_{i<n}A_i$ 
is small. 
\end{lemma}
\begin{proof}
Suppose that $A_i$ is small for every $i<n$. Fix an interval $I$. We aim to show that 
$A^n=\bigcup_{i<n}A_i$ is not dense in $I$. By $\Sigma^0_2$-induction we prove that for all $i\leq n$ there exists an interval $J\subseteq I$ such that $A^i\cap J=\emptyset$, where $A^i=\bigcup_{j<i}A_j$. For $i=n$ we have the desired conclusion. The case $i=0$ is trivial. Suppose $i+1\leq n$. By induction there exists an interval $J\subseteq I$ such that $A^i\cap J=\emptyset$.  By the assumption $A_i$ is small and so there exists an interval $K\subseteq J$ such that $A_i\cap K=\emptyset$. It follows that  $A^{i+1}\cap K=(A^i\cup A_i)\cap K=\emptyset$.
%By bounded $\Sigma^0_2$-comprehension we can form the set 
%\[ C=\{i\leq n\colon (\exists J\subseteq I\text{ interval})A^i\cap J=\emptyset\}. \]
%where $A^i=\bigcup_{j<i}A_j$. The set $C$ is nonempty as $0\in C$. Let $i$ be the maximum of $C$. We claim that $i=n$. Suppose not.  By definition there 
%exists an interval $J\subseteq I$ such that $A^i\cap J=\emptyset$. By the assumption $A_i$ is small and so there exists an interval $K\subseteq J$ such that $A_i\cap K=\emptyset$. It follows that 
%$A^{i+1}\cap K=(A^i\cup A_i)\cap K=\emptyset$, and hence $i+1\in C$, for the desired contradiction.  
\end{proof}

\begin{theorem}\label{upper}
$\erp$ is provable in $\aca$.
\end{theorem}
\begin{proof}
Let $f\colon [\Qb]^2\to 2$ be given. For any $x\in \Qb$, let 
$\Red(x)=\{y\in\Qb\setminus\{x\}\colon f(x,y)=0\}$. Define 
$\blue(x)$ accordingly. We say that $A\subseteq \Qb$ is
\emph{red-admissible} if there exists some $x\in A$ such that $A\cap\Red(x)$ is 
positive.

Case I. Every positive subset of $\Qb$ is red-admissible. We aim to show that 
there exists an infinite $0$-homogeneous set. We define by arithmetical recursion a 
sequence $(x_n)_{n\in\Nb}$ as 
follows.
Supppose we have defined $x_i$ for all $i<n$, and assume by arithmetical 
induction that $A_n=\bigcap_{i<n}\Red(x_i)$ is positive, and hence red-admissible (where $\bigcap_{i < 0} \Red(x_i) = \Qb$).  
Search for the $\omega$-least $x_n\in A_n$ such that 
$A_n\cap\Red(x_n)=\bigcap_{i<{n+1}}\Red(x_i)$ 
is positive. By definition, the set $\{x_n\colon n\in\Nb\}$ is infinite and
$0$-homogeneous.

Case II. There is a positive subset $A$ of $\Qb$ which is not 
red-admissible. In this case, we show that there exists a dense 
$1$-homogeneous set. Let $I$ be a witness of $A$ being positive. Fix an 
enumeration $(I_n)_{n\in\Nb}$ of all subintervals of $I$. Notice that by definition $A$ 
intersects every $I_n$.

We define by arithmetical recursion a sequence $(x_n)_{n\in\Nb}$ as follows. Let 
$x_0\in A\cap I_0$. Suppose we have defined $x_i\in A\cap I_i$ for all
$i<n$. By Lemma \ref{small}, since every $A\cap\Red(x_i)$ with $i<n$ is small, 
it follows that $E=\bigcup_{i<n} \big(A\cap\Red(x_i)\big)$ is small.
Let $J\subseteq I_n$ be such that $E\cap J=\emptyset$. We may safely assume that no 
$x_i$ with $i<n$ belongs to $J$. Since $A$ is dense in $I$ and $J\subseteq I$, we can 
find $x_n\in A\cap J$. In particular, $x_n\in
\bigcap_{i<n}\blue(x_i)$. Therefore $\{x_n\colon n\in\N\}$ is  dense and
$1$-homogeneous.
\end{proof}

\begin{remark}
A similar proof shows that $\rt^2_2$ is provable in $\aca$. In fact, we can consider the 
ideal of finite sets of $\Nb$ so that a positive set is just an infinite set and a 
red-admissible set is a set $A\subseteq\N$ such that $A\cap\Red(x)$ is infinite 
for some $x\in A$. 
\end{remark}

\section{Pigeonhole principle on $\Qb$}\label{sect:pigeonhole}

We next consider the statement $\ers$ asserting that every finite coloring of 
rationals has a dense homogeneous set.  The main result is that $\ers$ is stronger than $\bst$ over $\rca$. We 
achieve this by adapting the model-theoretic proof of Corduan, 
Groszek, and Mileti \cite{Corduan2010Reverse} that separates  $\tto^1$ from $\bst$.
Basically, in a model of 
$\rca+\neg\ist$, there are a real $X$ and an $X$-recursive instance of $\ers$ with no 
$X$-recursive solutions. Before going into the details of this proof, we establish the 
following simple reverse mathematics facts.

\begin{lemma}
Over $\rca$,
\begin{enumerate}[\quad $1)$]
\item $\erp\lor\ist\rightarrow\ers$
\item $\ers\rightarrow\rt^1$.
\end{enumerate}
\end{lemma}
\begin{proof}
$1)$ Let $f\colon\Qb\to n$ be a given coloring. First assume $\erp$, and let 
$g\colon[\Qb]^2\to 2$ be defined by $g(x,y)=1$ if and only if $f(x)=f(y)$. Provably in 
$\rca$ every one-to-one function from an infinite set is unbounded. Then by 
$\erp$ there exists a dense $1$-homogeneous set for $g$, which is homogeneous 
for $f$. 

Now assume $\ist$ and define $A_i=f^{-1}(i)$ for $i<n$. As $\Qb=\bigcup_{i<n}A_i$ 
is positive, by  lemma \ref{small}, 
there exists $i<n$ such that $A_i$ is positive. From $A_i$ we can compute a dense 
$i$-homogeneous set.  

$2)$ is trivial.
\end{proof}

As in \cite{Corduan2010Reverse}, the proof of our separation result consists 
of a few lemmas. We start by first adapting \cite[Lemma 
3.4]{Corduan2010Reverse} (see Lemma \ref{diagonal} below). The combinatorial core of the 
proof is based on the following.

\begin{lemma}[$\mathsf{I}\Sigma_1$]\label{disj}
For each~$e < n$, let $\Gamma_e$ consist of $4n$ 
pairwise disjoint intervals of $\Qb$. Then there exist $2n$ pairwise disjoint intervals 
$\langle I_{e,i}\colon e<n,i<2\rangle$ such that $I_{e,i}\in 
\Gamma_e$ for all $e<n$ and~$i<2$.
\end{lemma}
\begin{proof}
Let $\Gamma_e$, $e<n$, be given. Consider the following recursive procedure. At each 
stage we define $\Gamma_{e,s}$ for $e<n$ and $\Delta_s$ as follows.  At 
stage $0$, 
$\Gamma_{e,0}=\Gamma_e$ and $\Delta_0=\langle\rangle$. At stage $s+1$, if 
$|\Delta_s|=2n$ or $\Gamma^s=\bigcup_{e<n}\Gamma_{e,s}$ is empty, we are done. Otherwise 
search for $I\in \Gamma^{s}$  minimal with respect to inclusion (such an interval 
exists by $\mathsf{I}\Sigma_1$). Add $I$ to 
$\Delta_s$, that is, $\Delta_{s+1}=\Delta_s\conc I$. Let $e$ be such that $I\in 
\Gamma_{e,s}$.  If $\Delta_s$ already contains an interval in $\Gamma_{e}$, let 
$\Gamma_{e,s+1}=\emptyset$, otherwise let $\Gamma_{e,s+1}=\Gamma_{e,s}\setminus\{I\}$. 
For all $j\neq e$, let $\Gamma_{j,s+1}=\{J\in \Gamma_{j,s}\colon I\cap J=\emptyset\}$. 
Notice 
that by the choice of $I$ as minimal, at most two intervals from each $\Gamma_{j,s}$ with 
$j\neq e$ have nonempty intersection with $J$. 

By $\mathsf{I}\Sigma_1$ (indeed $\mathsf{I}\Sigma_0$) it is 
easy to show that, for all $s<2n+1$, $\Delta_s$ consists of $s$ disjoint intervals from 
$\bigcup_{e<n}\Gamma_e$ with at most two intervals from the same $\Gamma_e$, that every 
interval in $\Delta_s$ is disjoint from any interval in $\Gamma^s$, and that if $\Delta_s$ 
does not contain $2$ intervals from $\Gamma_e$, then $\Gamma_{e,s}$ contains at least 
$4n-2s$ intervals. In particular,  $\Delta_{2n}$ is as desired.
\end{proof}

\begin{lemma}[$\rca$]\label{diagonal}
For every real $X$ there exists an $X$-recursive function $d\colon \N\times\Qb\to 2$ 
such that for all $n$ and $e<n$, if $W_e^X$ is a dense set of $\Qb$, then there exist two 
disjoint intervals $I_0,I_1$ such that $W_e^X\cap I_i$ is infinite for all $i<2$ 
and $d(n,x)=i$ for all $i<2$ and for almost every $x\in I_i$.
\end{lemma}
\begin{proof}[Proof sketch]
Our strategy to defeat $n$-many dense sets $\{A_e\colon e<n\}$ 
is to choose $2n$ pairwise disjoint intervals $I_{e,0}, I_{e,1}$ for $e<n$ so that 
each $I_{e,i}$ has end-points in $A_e$, and assign color $i$ to the interval 
$I_{e,i}$ for all $e<n$ and $i<2$. As we want to diagonalize against $n$-many potential 
dense sets of the form $W_e^X$ for $e<n$ and we cannot decide uniformly in $n$
which ones are dense, we act only when some $W_e^X$ outputs 
$4n+1$ points. We then specify a set $\Gamma_e$ of $4n$ disjoint intervals with 
end-points in $W_e^X$ and from each $\Gamma_e$ currently defined we choose
intervals $I_{e,0}$ and $I_{e,1}$ as in Lemma \ref{disj}. Every time we act, our choice 
of $I_{e,0}$ and $I_{e,1}$ might change, but this happens at most $n$-many times. As the 
actual construction is essentially the one in the proof of  \cite[Lemma 
3.4]{Corduan2010Reverse}, we leave the details to the reader.
\end{proof}

The next lemma is the key part of the whole argument (see \cite[Proposition 
3.5]{Corduan2010Reverse}).

\begin{lemma}\label{diagonal2}
Let $M$ be a model of $\rca+\neg\is2$. Then for some real $X\in M$ there is an $X$-recursive 
(in the sense of $M$) coloring $f$ of $\Qb$ into $M$-finitely many colors such that no 
$X$-recursive dense set is homogeneous for $f$. 
\end{lemma}
\begin{proof}
Let $X\in M$ witness  the failure of $\is2$. Then there exists  an $X$-recursive function 
$h\colon\Nb^2\to\N$ 
such that for some number $a$, the range of the partial function 
$h(y)=\lim_{s\to\infty}h(y,s)$ is unbounded on $\{y\colon y<a\}$ (see also \cite[Lemma 
3.6]{Corduan2010Reverse}). Define 
$f\colon \Qb\to 2^a$ by 
\[  f(x)=\langle d(h(y,x),x)\colon y<a \rangle,
\]
where $d(n,x)$ is the function of Lemma \ref{diagonal}. Let $W_e^X$ be a dense set of 
$\Qb$. We aim to show that $W_e^X$ is not homogeneous for $f$. Let $y<a$ such that 
$h(y)>e$. Observe that for almost every $x\in\Qb$ the $y$th bit of $f(x)$ is 
$d(h(y),x)$. As $e<h(y)$, let $I_0$ and $I_1$ be two intervals as in Lemma 
\ref{diagonal}. Now for sufficiently large  $x_0\in W_e^X\cap I_0$ and $x_1\in 
W_e^X\cap I_1$ we have $d(h(y,x_i),x_i)=d(h(y),x_i)=i$, and hence $f(x_0)\neq f(x_1)$.
\end{proof}

We can finally prove the analogue of \cite[Corollary 3.8]{Corduan2010Reverse}, which is 
the main result.

\begin{theorem}
Let $P$ be a $\Pi^1_1$ sentence. Then $\rca+P\vdash \ers$ if and only if 
$\rca+P\vdash\ist$. In particular, $\rca+\bst\not\vdash\ers$.
\end{theorem}
\begin{proof}
The argument is the same as in the proof of \cite[Theorem 3.7]{Corduan2010Reverse}. As 
$\rca+\ist\vdash\ers$, we just need to prove one implication. Suppose that
$\rca+P\not\vdash\ist$, and let $M$ be a model of $\rca+P$ where $\ist$ fails. By Lemma \ref{diagonal2}, for some real $X\in M$, there exists an $X$-recursive instance of 
$\ers$ with no $X$-recursive solutions. Let $M'$ be the submodel of 
$M$ with the same 
first-order part as $M$ and second-order part consisting of the reals recursive in $X$ 
(in the 
sense of $M$). Therefore $\ers$ fails in $M'$.  Since $M'$ has same first-order part as $M$, $M'$ satisfies the $\Pi^1_1$ 
sentence $P$. As the reals of $M'$ are the ones recursive in a given real of $M$, $M$ satisfies 
$\rca$. Thus $\rca+P\not\vdash\ers$. 
\end{proof}

\section{$\er^2_2$ does not computably reduce to~$\rt^2_2$}\label{sect:separation}

Many proofs of $\Qsf \imp \Psf$ over~$\rca$ make use only of one $\Qsf$-instance
to solve a $\Psf$-instance. This is the notion of computable reducibility.

\begin{definition}[Computable reducibility] Fix two~$\Pi^1_2$ statements~$\Psf$ and $\Qsf$.
$\Psf$ is \emph{computably reducible} to~$\Qsf$ (written $\Psf \leq_c \Qsf$)
if every~$\Psf$-instance~$X_0$ computes a~$\Qsf$-instance~$X_1$ such that for every 
solution~$Y$ to~$X_1$,  $Y \oplus X_0$ computes a solution to~$X_0$.
\end{definition}

Proving that~$\Psf \leq_c \Qsf$ is not sufficient to deduce that~$\rca \vdash \Qsf \imp \Psf$.
One needs to prove that this reduciblity can be formalized within $\rca$, and in 
particular
that $\Sigma^0_1$-induction is sufficient to prove its validity.
The fine-grained nature of computable reducibility enables one to exhibit
distinctions between statements which would not have been revealed in reverse mathematics.
For example, $\rt^2_k$ and~$\rt^2_{k+1}$ are equivalent over~$\rca$ 
whereas~$\rt^2_{k+1} \not \leq_c \rt^2_k$~\cite{Patey2015weakness}.

This notion of reducibility can be also seen as an intermediary
step to tackle difficult separations~\cite{DzhafarovStrong}.
Proving that $\Psf \not \leq_c \Qsf$ is simpler than
separating $\Qsf$ from~$\Psf$ over $\omega$-models.
Lerman, Solomon and Towsner~\cite{Lerman2013Separating} 
introduced a framework to separate Ramsey-type statements
over~$\omega$-models, in which they transform a one-step diagonalization,
that is, computable non-reducibility, into a separation
in the sense of reverse mathematics.
In this section, we prove that the Erd\H{o}s-Rado theorem for pairs
does not reduce to Ramsey's theorem for pairs in one step.

\begin{theorem}\label{thm:er22-not-computably-reduces}
$\erp \not \leq_c \rt^2$.
\end{theorem}

Interestingly, this diagonalization does not seem to be easily
generalizable to a separation over~$\omega$-models.
A reason is that the fairness property  ensured by the~$\erp$-instance
does not seem to be preserved by weak K\"onig's lemma.
This is hitherto the first example of a computable non-reducibility
of a principle~$\Psf$ to~$\rt^2$ which is not generalizable
to a proof that~$\rt^2_2$ does not imply~$\Psf$ over~$\rca$.

The remainder of this section is devoted to a proof of Theorem~\ref{thm:er22-not-computably-reduces}.
The notion of fairness presented below may have some ad-hoc flavor.
It has been obtained by applying the main ideas of the framework
of Lerman, Solomon and Towsner~\cite{Lerman2013Separating,Patey2015Iterative}.
Thanks to an analysis of the combinatorics of Ramsey's theorem for pairs
and the Erd\H{o}s-Rado theorem for pairs, we prove our
computable non-reducibility result by constructing an instance of~$\erp$
ensuring the density of the diagonalizing conditions in the forcing notion of~$\rt^2_2$.
Then we abstract the diagonalization to 
any $\Sigma^0_1$ formula, to get rid of the specificities
of the forcing notion of~$\rt^2_2$ in the notion of fairness preservation.
See~\cite{Patey2015strength} for a detailed example of the various steps of this framework,
leading to a separation of~$\rt^2_2$ from the tree theorem for pairs over~$\rca$.

\begin{definition}[Simple partition]
A \emph{simple partition}~$\inter_\Qb(S)$ is a finite sequence of open intervals~$(-\infty, x_0), (x_0, x_1), \dots, (x_{n-1}, +\infty)$
for some set of rationals~$S = \{x_0 <_\Qb \dots <_\Qb x_{n-1}\}$. We set 
$\inter_\Qb(\emptyset) = \{\Qb\}$. 
A simple partition~$I_0, \dots, I_{n-1}$
\emph{refines} another simple partition~$J_0, \dots, J_{m-1}$ if for every~$i < n$,
there is some~$j < m$ such that~$I_i \subseteq J_j$.
Given two simple partitions~$I_0, \dots, I_{n-1}$ and~$J_0, \dots, J_{m-1}$,
the product~$\vec{I} \otimes \vec{J}$ is the simple partition
\[
\{ I \cap J : I \in \vec{I} \wedge J \in \vec{J} \}
\]
\end{definition}
One can easily see that~$\inter_\Qb(S)$ refines~$\inter_\Qb(T)$ if~$T \subseteq S$
and that $\inter_\Qb(S\cup T)=\inter_\Qb(S)\otimes\inter_\Qb(T)$.
Note that every simple partition has a finite description, since the set~$S$
and each rational has a finite description.
Also note that a simple partition is not a true partition of~$\Qb$ since the endpoints do not belong
to any interval. However, we have~$S \cup \bigcup \inter_\Qb(S) = \Qb$.

\begin{definition}[Matrix]
An \emph{$m$-by-$n$ matrix} $M$ is a rectangular array of rationals $x_{i,j} \in \Qb$
such that $x_{i,j} <_\Qb x_{i,k}$ for each~$i < m$ and $j < k < 
n$. 
The $i$th \emph{row} $M(i)$ of the matrix $M$ is the $n$-tuple of rationals $x_{i,0} < 
\dots < x_{i,n-1}$. The simple partition $\inter_\Qb(M)$ is defined by  $\bigotimes_{i < 
m} \inter_\Qb(M(i))$. In particular, $\bigotimes_{i <m} \inter_\Qb(M(i))$ refines the 
simple partition~$\inter_\Qb(M(i))$ for each~$i < m$.
\end{definition}

It is important to notice that an $m$-by-$n$ matrix is formally a 3-tuple~$\tuple{m,n, M}$
and not only the matrix itself~$M$. This distinction becomes important
when dealing with the degenerate cases. An~$m$-by-0 matrix~$M$
and a 0-by-$n$ matrix~$N$ are both empty.  However, they have different sizes.
In particular, we shall define the notion of~$M$-type for a matrix, and this definition
will depend on the number of columns of the matrix~$M$, which is~0 for~$M$, and~$n$ for~$N$.
Notice also that, for a degenerate matrix $M$, the simple partition $\inter_{\Qb}(M)$ 
is the singleton $\{\Qb\}$. 

Given a simple partition~$\vec{I}$, we want to classify the $k$-tuples of rationals
according to which interval of~$\vec{I}$ they belong to.
This leads to the notion of~$(\vec{I},k)$-type.

\begin{definition}[Type]
Given a simple partition~$I_0, \dots, I_{n-1}$ and some~$k \in \omega$,
an \emph{$(\vec{I},k)$-type} is a tuple~$T_0, \dots, T_{k-1}$
such that~$T_i \in \vec{I}$ for each~$i < k$.
Given an $m$-by-$n$ matrix $M$, an~\emph{$M$-type} is an~$(\inter_\Qb(M), n)$-type.
\end{definition}

We now state two simple combinatorial lemmas which will be useful later.
The first trivial lemma simply states that each $m$-tuple of rationals (different from
the endpoints of a simple partition) belongs to a type.

\begin{lemma}\label{lem:tuple-has-mtype}
For every simple partition $I_0, \dots, I_{n-1}$ and every~$k$-tuple
of rationals~$x_0, \dots, x_{k-1} \in \bigcup_{i < n} I_i$, there is an $(\vec{I},k)$-type $T_0, \dots, T_{k-1}$
such that~$x_j \in T_j$ for each~$j < k$.
\end{lemma}
\begin{proof}
Fix $k$ rationals $x_0, \dots, x_{k-1}$. For each~$i < k$, 
there is some interval~$T_i \in \vec{I}$ such that~$x_i \in T_i$ since $x_i \in \bigcup_{j < n} 
I_j$. The sequence $T_0, \dots, T_{k-1}$ is the desired~$(\vec{I}, k)$-type.
\end{proof}

The next lemma is a consequence of the pigeonhole principle.

\begin{lemma}\label{lem:mtype-interval-disjoint}
For every $m$-by-$n$ matrix~$M$ and every $M$-type $T_0, \dots, T_{n-1}$,
there is an $m$-tuple of intervals~$J_0, \dots, J_{m-1}$ with~$J_i \in \inter_\Qb(M(i))$ such that
\[
(\bigcup_{j < n} T_j) \cap (\bigcup_{i < m} J_i) = \emptyset
\]
\end{lemma}
\begin{proof}
Let $T_0, \dots, T_{n-1}$ be an~$M$-type.
For every~$i < m$ and~$j < n$, there is some $J \in \inter_\Qb(M(i))$ such that~$T_j \subseteq J$.
Since $|\inter_\Qb(M(i))| = n+1$, there is an interval~$J_i \in \inter_\Qb(M(i))$ such that $(\bigcup_{j < n} T_j) \cap J_i = \emptyset$.
\end{proof}

\begin{definition}[Formula, valuation]
Given an $m$-by-$n$ matrix $M$, an \emph{$M$-formula} 
is a formula $\varphi(\vec{U}, \vec{V})$ with distinguished (finite coded) set variables $U_j$ for each $j < n$
and~$V_{i,I}$ for each~$i < m$ and~$I \in \inter_\Qb(M(i))$.
An \emph{$M$-valuation $(\vec{R}, \vec{S})$} is a tuple of finite sets $R_j \subseteq \Qb$ for each~$j < n$
and~$S_{i,I} \subseteq I$ for each~$i < m$ and~$I \in \inter_\Qb(M(i))$.
The $M$-valuation~$(\vec{R}, \vec{S})$ is of type~$\vec{T}$ for some~$M$-type~$T_0, \dots, T_{n-1}$
if moreover $R_j \subseteq T_j$ for each~$j < n$.
The $M$-valuation $(\vec{R}, \vec{S})$ \emph{satisfies} $\varphi$ if $\varphi(\vec{R}, \vec{S})$ holds.
\end{definition}

Given some valuation $(\vec{R}, \vec{S})$ and some integer $s$, we write $(\vec{R}, \vec{S}) > s$
to say that for every $x \in (\bigcup \vec{R}) \cup (\bigcup \vec{S})$, $x > s$. 
Following the terminology of~\cite{Lerman2013Separating}, we define 
the notion of essentiality for a formula (an abstract requirement),
which corresponds to the idea that there is room for diagonalization
since the formula is satisfied by valuations which are arbitrarily far.

\begin{definition}[Essential formula]
Given an $m$-by-$n$ matrix~$M$, an $M$-formula $\varphi$
is \emph{essential} if for every $s \in \omega$,
there are an $M$-type $\vec{T}$ and an $M$-valuation $(\vec{R}, \vec{S}) > s$ of type~$\vec{T}$
such that $\varphi(\vec{R}, \vec{S})$ holds.
\end{definition}

The notion of fairness is defined accordingly. If some formula
is essential, that is, leaves enough room for diagonalization, then there is
an actual valuation which will diagonalize against the~$\erp$-instance.

\begin{definition}[Fairness]
Fix two sets $A_0, A_1 \subseteq \Qb$.
Given an $m$-by-$n$ matrix $M$, an $M$-valuation~$(\vec{R}, \vec{S})$ 
\emph{diagonalizes} against $A_0, A_1$
if $\bigcup \vec{R} \subseteq A_1$ and for every $i < m$, there is some~$I \in \inter_\Qb(M(i))$ 
such that $S_{i,I} \subseteq A_0$.
A set~$X$ is \emph{fair} for~$A_0, A_1$ if for every $m, n \in \omega$, every $m$-by-$n$ matrix~$M$
and every $\Sigma^{0,X}_1$ essential $M$-formula, there is an $M$-valuation $(\vec{R}, \vec{S})$ diagonalizing against $A_0, A_1$ 
such that $\varphi(\vec{R}, \vec{S})$ holds.
\end{definition}

Of course, if $Y \leq_T X$, then every $\Sigma^{0,Y}_1$ formula is $\Sigma^{0,X}_1$.
As an immediate consequence, if $X$ is fair for some $A_0, A_1$ and $Y \leq_T X$, then $Y$ is fair for $A_0, A_1$.

Now that we have introduced the necessary terminology, we create a non-effective
instance of~$\erps$ which will serve as a bootstrap for fairness preservation. Remember that $erps$ asserts that for every partition $A_0\cup A_1=\Qb$ of the rationals there exists either 
an infinite subset of $A_0$ or a dense subset of $A_1$.

\begin{lemma}\label{lem:partition-emptyset-er-fair}
For every set~$C$, there exists a $\Delta^{0,C}_2$ partition $A_0 \cup A_1 = \Qb$ such that
$C$ is fair for~$A_0, A_1$.
\end{lemma}
\begin{proof}
The proof is  a no-injury priority construction.
Let $M_0, M_1, \dots$ be an enumeration of all $m$-by-$n$ matrices
and $\varphi_0, \varphi_1, \dots$ be an effective enumeration of all $\Sigma^{0,C}_1$ $M_k$-formulas for every $m, n \in \omega$.
We want to satisfy the following requirements for each pair of integers~$e,k$.

\begin{quote}
$\Rcal_{e,k}$: If the~$M_k$-formula $\varphi_e$ is essential, then $\varphi_e(\vec{R}, \vec{S})$ holds
for some $M_k$-valuation $(\vec{R}, \vec{S})$ diagonalizing against $A_0, A_1$.
\end{quote}

The requirements are ordered via the standard pairing function $\tuple{\cdot, \cdot}$.
The sets $A_0$ and $A_1$ are constructed by a $C'$-computable list of 
finite approximations $A_{i,0} \subseteq A_{i,1} \subseteq \dots$
such that all elements added to~$A_{i,s+1}$ from~$A_{i,s}$
are strictly greater than the maximum of~$A_{i,s}$ (in the~$\Nb$ order) for each~$i < 2$. 
We then let $A_i = \bigcup_s A_{i,s}$ which will be a~$\Delta^{0,C}_2$ set.
At stage 0, set $A_{0,0} = A_{1,0} = \emptyset$. Suppose that at stage $s$,
we have defined two disjoint finite sets $A_{0,s}$ and $A_{1,s}$ such that
\begin{itemize}
	\item[(i)] $A_{0,s} \cup A_{1,s} = [0,b]_\Nb$ for some integer $b \geq s$
	\item[(ii)] $\Rcal_{e',k'}$ is satisfied for every $\tuple{e',k'} < s$
\end{itemize}
Let $\Rcal_{e,k}$ be the requirement such that $\tuple{e,k} = s$.
Decide $C'$-computably whether there are some
$M_k$-type~$\vec{T}$ and some $M_k$-valuation $V=(\vec{R}, \vec{S}) > b$ of type~$\vec{T}$
such that $\varphi_e(V)$ holds. If so, $C$-effectively fetch~$\vec{T} = T_0, \dots, T_{n-1}$
and such a $(\vec{R}, \vec{S}) > b$. Let $d$ be an upper bound (in the~$\Nb$ order) on the rationals in $(\vec{R}, \vec{S})$.
By Lemma~\ref{lem:mtype-interval-disjoint}, for each~$i < m$,
there is some~$J_i \in \inter_\Qb(M(i))$ such that
\[
(\bigcup_{j < n} T_j) \cap (\bigcup_{i < m} J_i) = \emptyset
\]
Set $A_{0,s+1} = A_{0,s}\cup \bigcup_{i < m} J_i \cap (b,d]_\Nb$ and $A_{1,s+1} = [0,d]_\Nb \setminus A_{0,s+1}$.
This way, $A_{0,s+1} \cup A_{1,s+1} = [0, d]_\Nb$.
By the previous equation, $\bigcup_{j < n} T_j \cap (b, d]_\Nb \subseteq [0,d]_\Nb \setminus A_{0,s+1}$ 
and the requirement $\Rcal_{e,k}$ is satisfied.
If no such $M_k$-valuation is found, the requirement $\Rcal_{e,k}$ is vacuously satisfied.
Set $A_{0,s+1} = A_{0,s} \cup \{b+1\}$ and $A_{1,s+1} = A_{1,s}$.
This way, $A_{0,s+1} \cup A_{1,s+1} = [0, b+1]_\Nb$.
In any case, go to the next stage. This finishes the construction.
\end{proof}

\begin{lemma}\label{lem:er-fair-not-solution}
If~$X$ is fair for some sets~$A_0, A_1 \subseteq \Qb$,
then $X$ computes neither an infinite subset of~$A_0$,
nor a dense subset of~$A_1$.
\end{lemma}
\begin{proof}
Since fairness is downward-closed under  Turing reducibility,
it suffices to prove that if~$X$ is infinite and fair for~$A_0, A_1$,
then it intersects both~$A_0$ and~$A_1$.

We first prove that~$X$ intersects~$A_1$.
Let~$M$ be the 0-by-1 matrix and~$\varphi(U)$ be the
$M$-formula which holds if~$U \cap X \neq \emptyset$. Note that $\varphi(U)$
is $\Sigma^{0,X}_1$ since $U$ is a finite coded set.
The only~$M$-type is~$\Qb$ and since $X$ is infinite, $\varphi$ is essential.
By fairness of~$X$, there is an $M$-valuation~$R$ diagonalizing against~$A_0, A_1$
such that~$\varphi(R)$ holds. By definition of diagonalization, $R \subseteq A_1$.
Since~$R \cap X \neq \emptyset$, this shows that~$X \cap A_1 \neq \emptyset$.

We now prove that~$X$ interects~$A_0$.
Let~$M$ be the~1-by-0 matrix and~$\varphi(V)$ be the~$\Sigma^{0,X}_1$
$M$-formula which holds if~$V \cap X \neq \emptyset$.
The $M$-formula $\varphi$ is essential since $X$ is infinite.
By fairness of~$X$, there is an $M$-valuation~$S$ diagonalizing against~$A_0, A_1$
such that~$\varphi(S)$ holds. By definition of diagonalization, $S \subseteq A_0$.
Since~$S \cap X \neq \emptyset$, this shows that~$X \cap A_0 \neq \emptyset$.
\end{proof}

Note that we did not use the fact that~$X$ is dense to make sure it intersects $A_0$.
Density will be useful in the proof of Theorem~\ref{thm:rt22-not-er-fairness}.

\begin{definition}
A \emph{Scott set} is a set~$\Scal \subseteq 2^{\omega}$ such that
\begin{itemize}
	\item[(i)] $(\forall X \in \Scal)(\forall Y \leq_T X)[Y \in \Scal]$
	\item[(ii)] $(\forall X, Y \in \Scal)[X \oplus Y \in \Scal]$
	\item[(iii)] Every infinite, binary tree in $\Scal$ has an infinite path in~$\Scal$.
\end{itemize}
\end{definition}

\begin{theorem}\label{thm:rt22-not-er-fairness}
Let~$A_0, A_1 \subseteq \Qb$ and~$\Scal$ be a Scott set whose members are all fair for~$A_0, A_1$.
For every set~$C \in \Scal$, every~$C$-computable coloring~$f : [\omega]^2 \to k$,
there is an infinite $f$-homogeneous set~$H$ such that~$H \oplus C$ computes neither
an infinite subset of~$A_0$, nor a dense subset of~$A_1$.
\end{theorem}
\begin{proof}
The proof is by induction over the number of colors~$k$.
The case~$k = 1$ is ensured by Lemma~\ref{lem:er-fair-not-solution}.
Fix a set~$C \in \Scal$ and let $f : [\omega]^2 \to k$ be a $C$-computable coloring.
If~$f$ has an infinite $f$-thin set $H \in \Scal$, that is, an infinite set over which $f$
avoids at least one color, then $H \oplus C$
computes a coloring~$g : [\omega]^2 \to k-1$ such that every infinite $g$-homogeneneous
set computes relative to~$H \oplus C$ an infinite $f$-homogeneous set. Since~$H \oplus C \in \Scal$,
by induction hypothesis, there is an infinite $g$-homogeneous set~$H_1$
such that~$H_1 \oplus H \oplus C$ computes neither an infinite subset of~$A_0$, nor a dense subset of~$A_1$.
So suppose that $f$ has no infinite $f$-thin set in~$\Scal$.

We construct $k$ infinite sets~$G_0, \dots, G_{k-1}$.
We need therefore to satisfy the following requirements for each~$p \in \omega$.
\[
  \Ncal_p : \hspace{20pt} (\exists q_0 > p)[q_0 \in G_0] 
		\hspace{20pt} \wedge \dots \wedge \hspace{20pt} 
	(\exists q_{k-1} > p)[q_{k-1} \in G_{k-1}] 
\]
Furthermore, we want to ensure that one of the~$G$'s computes
neither an infinite subset of~$A_0$, nor a dense subset of~$A_1$.
To do this, we will satisfy the following requirements
for every $k$-tuple of integers $e_0, \dots, e_{k-1}$.
\[
  \Qcal_{\vec{e}} : \hspace{20pt} 
		\Rcal_{e_0}^{G_0} \hspace{20pt} \vee \dots \vee \hspace{20pt} \Rcal_{e_{k-1}}^{G_{k-1}}
\]
where $\Rcal_e^H$ holds if $W^{H \oplus C}_e$ is neither an infinite subset of~$A_0$, nor a dense subset of~$A_1$.

We construct our sets $G_0, \dots, G_{k-1}$ by forcing. 
Our conditions are variants of Mathias conditions~$(F_0, \dots, F_{k-1}, X)$
such that each~$X$ is an infinite set in~$\Scal$, each $F_i$ is a finite set with $\max(F_i)<\min(X)$, and the following property holds:
\begin{itemize}
	\item[(P)] $(\forall i < k)(\forall x \in X)[F_i \cup \{x\} \mbox{ is } f\mbox{-homogeneous with color } i]$
\end{itemize}
A condition~$d = (E_0, \dots, E_{k-1}, Y)$ \emph{extends} $c = (F_0, \dots, F_{k-1}, X)$
if $(E_i, Y)$ Mathias extends $(F_i, X)$ for every~$i < k$.
We now prove the progress lemma, stating that we can force the~$G$'s to be infinite.
This is where we use the fact that there is no infinite $f$-thin set in~$\Scal$.

\begin{lemma}\label{lem:rt12-er-fairness-progress}
For every condition~$c = (F_0, \dots, F_{k-1}, X)$, every $i < k$ and every~$p \in \omega$
there is some extension~$d = (E_0, \dots, E_{k-1}, Y)$ such that~$E_i \cap (p,+\infty)_\Nb \neq \emptyset$.
\end{lemma}
\begin{proof}
Fix~$c$, $i$ and~$p$. If for every~$x \in X \cap (p,+\infty)_\Nb$
and almost every~$y \in X$, $f(x,y) \neq i$, then $X$ computes an infinite
$f$-thin set, contradicting our hypothesis. Therefore, there is some~$x \in X \cap (p,+\infty)_\Nb$
such that~$f(x, y) = i$ for infinitely many~$y \in X$.
Let~$Y$ be the collection of such $y$'s. The condition~$(F_0, \dots, F_{i-1}, F \cup \{x\}, F_{i+1}, \dots, F_k, Y)$
is the desired extension.
\end{proof}

We now prove the core lemma stating that we can satisfy each $\Qcal$-requirement.
A condition~$c$ \emph{forces} a requirement~$\Qcal$
if $\Qcal$  holds for every set~$G$ satisfying~$c$.

\begin{lemma}\label{lem:rt12-er-fairness-forcing}
For every condition~$c = (F_0, \dots, F_{k-1}, X)$ and every $k$-tuple of indices~$\vec{e}$,
there is an extension~$d = (E_0, \dots, E_{k-1}, Y)$ forcing~$\Qcal_{\vec{e}}$.
\end{lemma}
\begin{proof}
We can assume that~$W^{F_i \oplus C}_{e_i}$ has already outputted at least $k$ elements
and is either included in~$A_0$ or in~$A_1$ for each~$i < k$.
Indeed, if~$c$ has no extension satisfying this condition, then
$c$ forces $W^{G_i \oplus C}_{e_i}$ to be finite or not to be a valid solution for 
some~$i< k$ 
and therefore forces~$\Qcal_{\vec{e}}$. For each~$i < k$, we associate the label~$\ell_i < 2$ and the number~$p_i$
such that~$W^{F_i \oplus C}_{e_i}$ is the~$(p_i+1)$th set of this form included in~$A_{\ell_i}$.

Let~$n$ be the number of sets~$W^{F_i \oplus C}_{e_i}$ which are included in~$A_0$,
and let~$M$ be the~$(k-n)$-by-$n$ matrix such that the~$j$th row is composed
of the~$n$ first elements already outputted by the set $W^{F_i \oplus C}_{e_i}$
where~$p_i = j$ and~$\ell_i = 1$. In other words, $M(j)$ are the $n$ first elements 
outputted by the $j$th set $W^{F_i \oplus C}_{e_i}$ included in~$A_1$.

Let~$\varphi(\vec{U}, \vec{V})$ be the $\Sigma^{0,X \oplus C}_1$ formula
which holds if for every $k$-partition $Z_0 \cup \dots \cup Z_{k-1} = X$,
there are some~$i < k$ and some finite set~$E \subseteq Z_i$
which is $f$-homogeneous with color~$i$ and such that either~$\ell_i = 0$
and $W^{(F_i \cup E) \oplus C}_{e_i} \cap U_{p_i} \neq \emptyset$,
or $\ell_i = 1$ and~$W^{(F_i \cup E) \oplus C}_{e_i} \cap V_{p_i, I} \neq \emptyset$ for each~$I \in \inter_\Qb(M(p_i))$.
We have two cases.

In the first case, $\varphi(\vec{U}, \vec{V})$ is essential.
Since $X \oplus C$ is fair for~$A_0, A_1$, there is 
an $M$-valuation~$(\vec{R}, \vec{S})$ diagonalizing against~$A_0, A_1$
such that~$\varphi(\vec{R}, \vec{S})$ holds.
By compactness and definition of diagonalization against~$A_0, A_1$, 
there is a finite subset~$D \subset X$ such that
for every $k$-partition $D_0 \cup \dots \cup D_{k-1} = D$,
there are some $i < k$ and some finite set~$E \subseteq D_i$
which is $f$-homogeneous with color~$i$ and such that either~$\ell_i = 0$
and $W^{(F_i \cup E) \oplus C}_{e_i} \cap A_1 \neq \emptyset$,
or $\ell_i = 1$ and~$W^{(F_i \cup E) \oplus C}_{e_i} \cap A_0 \neq \emptyset$.

Each~$y \in X \setminus D$ induces a $k$-partition~$D_0 \cup \dots \cup D_{k-1}$ of~$D$
by setting~$D_i = \{ x \in D : f(x, y) = i \}$. Since there are finitely many possible 
$k$-partitions of~$D$, there are a $k$-partition~$D_0 \cup \dots \cup D_{k-1} = D$
and an infinite $X$-computable set~$Y \subseteq X$ such that
\[
(\forall i < k)(\forall x \in D_i)(\forall y \in Y)[f(x,y) = i]
\]
We furthermore assume that~$min(Y)$ is larger than the use of the computations.
Let~$i < k$ and~$E \subseteq D_i$ be the~$f$-homogeneous set with color~$i$
such that either~$\ell_i = 0$ and $W^{(F_i \cup E) \oplus C}_{e_i} \cap A_1 \neq \emptyset$,
or $\ell_i = 1$ and~$W^{(F_i \cup E) \oplus C}_{e_i} \cap A_0 \neq \emptyset$.
The condition~$(F_0, \dots, F_{i-1}, F_i \cup E, F_{i+1}, \dots, F_{k-1}, Y)$
is an extension of~$c$ forcing $\Qcal_{\vec{e}}$ by the $i$th side.

In the second case, there is some threshold~$s \in \omega$
such that for every~$M$-type~$\vec{T}$, there is no $M$-valuation~$(\vec{R}, \vec{S}) > s$
of type~$\vec{T}$ such that~$\varphi(\vec{R}, \vec{S})$ holds.
By compactness, it follows that 
for every $M$-type~$\vec{T}$, the $\Pi^{0,X \oplus C}_1$ class~$\Ccal_{\vec{T}}$
of all $k$-partitions $Z_0 \cup \dots \cup Z_{k-1} = X$
such that for every~$i < k$ and every finite set~$E \subseteq Z_i$
which is $f$-homogeneous with color~$i$, either~$\ell_i = 0$
and $W^{(F_i \cup E) \oplus C}_{e_i} \cap T_{p_i} \cap (s,+\infty)_\Nb = \emptyset$,
or $\ell_i = 1$ and~$W^{(F_i \cup E) \oplus C}_{e_i} \cap I \cap (s,+\infty)_\Nb = \emptyset$
for some~$I \in \inter_\Qb(M(p_i))$ is non-empty.
Since $\Scal$ is a Scott set, for each~$M$-type $\vec{T}$, there is a 
$k$-partition~$\vec{Z}^{\vec{T}}\in \Ccal_{\vec{T}}$
such that~$\bigoplus_{\vec{T}} \vec{Z}^{\vec{T}} \oplus X \oplus C \in \Scal$.

If there are some~$M$-type~$\vec{T}$ and some~$i < k$ such that $\ell_i = 1$ 
and~$Z^{\vec{T}}_i$ is infinite, then the condition~$(F_0, \dots, F_{k-1}, Z^{\vec{T}}_i)$
extends $X$ and forces~$W^{G_i \oplus C}_{e_i}$ not to be dense. 
So suppose that it is not the case. Let~$Y \in \Scal$
be an infinite subset of~$X$ such that for each~$M$-type~$\vec{T}$,
there is some~$i < k$ such that~$Y \subseteq Z_i^{\vec{T}}$. 
Note that by the previous assumption, $\ell_i = 0$ for every such~$i$.
We claim that the condition~$(F_0, \dots, F_{k-1}, Y)$ forces $W^{G_i \oplus C}_{e_i}$ to be finite
for some~$i < k$ such that~$\ell_i = 0$. Suppose for the sake of contradiction that there are some
rationals~$x_0, \dots, x_{n-1} > s$ such that $x_{p_i} \in W^{G_i \oplus C}_{e_i}$ for each~$i < k$ where~$\ell_i = 0$.
Since~$x_0, \dots, x_{n-1} > s$, $x_0, \dots, x_{n-1} \in \bigcup \inter_\Qb(M)$. 
Therefore, by Lemma~\ref{lem:tuple-has-mtype}, let~$\vec{T}$ be the unique $M$-type
such that~$x_j \in T_j$ for each~$j < n$. By assumption, there is some~$i < k$
such that~$Y \subseteq Z^{\vec{T}}_i$ and~$\ell_i = 0$. By definition of~$Z^{\vec{T}}_i$,
$W^{G_i \oplus C}_{e_i} \cap T_{p_i} \cap (s,+\infty)_\Nb = \emptyset$,
contradicting~$x_{p_i} \in W^{G_i \oplus C}_{e_i}$.
\end{proof}

Using Lemma~\ref{lem:rt12-er-fairness-progress} and Lemma~\ref{lem:rt12-er-fairness-forcing}, define an infinite descending sequence 
of conditions~$c_0 = (\emptyset, \dots, \emptyset, \omega) \geq c_1 \geq \dots$
such that for each~$s \in \omega$
\begin{itemize}
	\item[(i)] $|F_{i,s}| \geq s$ for each~$i < k$
	\item[(ii)] $c_{s+1}$ forces~$\Qcal_{\vec{e}}$ if~$s = \tuple{e_0, \dots, e_{k-1}}$
\end{itemize}
where~$c_s = (F_{0,s}, \dots, F_{k-1,s}, X_s)$. Let~$G_i = \bigcup_s F_{i,s}$ for each~$i < k$.
The~$G$'s are all infinite by (i) and $G_i$ does not compute an $\erps$-solution to the~$A$'s for some~$i < k$ by (ii).
This finishes the proof of Theorem~\ref{thm:rt22-not-er-fairness}.
\end{proof}

We are now ready to prove the main theorem.

\begin{proof}[Proof of Theorem~\ref{thm:er22-not-computably-reduces}]
By the low basis theorem~\cite{Jockusch197201}, there is a low set $P$ of PA degree.
By Scott~\cite{Scott1962Algebras}, every PA degree bounds a Scott set. 
Let~$\Scal$ be a Scott set such that~$X \leq_T P$ for every~$X \in \Scal$.
By Lemma~\ref{lem:partition-emptyset-er-fair}, there is a $\Delta^{0,P}_2$ (hence $\Delta^0_2$)
partition $A_0 \cup A_1 = \Qb$ such that $P$ is fair for~$A_0, A_1$. In particular,
every set~$X \in \Scal$ is fair for~$A_0, A_1$ since fairness is downward-closed under the Turing reducibility.

By Schoenfield's limit lemma~\cite{Shoenfield1959degrees}, there is a computable function $h : [\Qb]^2 \to 2$ such 
that for each~$x \in \Qb$, $\lim_s h(x, s)$ exists and $x \in A_{\lim_s h(x, s)}$.
Note that for every infinite set $D$ 0-homogeneous for~$h$, $D \subseteq A_0$,
and for every dense set $D$ 1-homogeneous for~$h$, $D \subseteq A_1$.

Fix a computable $\rt^2$-instance~$f : [\omega]^2 \to k$. In particular, $f \in 
\Scal$.
By Theorem~\ref{thm:rt22-not-er-fairness}, there is an infinite $f$-homogeneous set~$H$
such that $H$ computes neither an infinite subset of~$A_0$, nor a dense subset of~$A_1$.
Therefore, $H$ computes no $\erp$-solution to~$h$.
\end{proof}

\section{Discussion and questions}\label{sect:questions}

This Erd\H{o}s-Rado theorem shares an essential feature with
another strengthening of Ramsey's theorem for pairs already studied in reverse mathematics:
the tree theorem for pairs~\cite{Chubb2009Reverse, Corduan2010Reverse, Dzhafarov2010Ramseys, Patey2015strength}.

\begin{definition}[Tree theorem]
We denote by $[2^{<\Nb}]^n$ the collection of \emph{linearly ordered} subsets of~$2^{<\Nb}$ of size~$n$.
A set $S \subseteq 2^{<\Nb}$ is \emph{order isomorphic} to $2^{<\Nb}$ (written $S \cong 2^{<\Nb}$) 
if there is a bijection $g : 2^{<\Nb} \to S$ such that for all $\sigma, \tau \in 2^{<\Nb}$,
$\sigma \preceq \tau$ if and only if $g(\sigma) \preceq g(\tau)$.
Given a coloring $f : [2^{<\Nb}]^n \to k$, a tree $S$ is $f$-homogeneous if $S \cong 2^{<\Nb}$
and $f \uh [S]^n$ is monochromatic.
$\tto^n_k$ is the statement ``Every coloring $f : [2^{<\Nb}]^n \to k$ has an $f$-homogeneous tree.''
\end{definition}

Both~$\tto^2_2$ and~$\erp$ lie between the arithmetic comprehension axiom and~$\rt^2_2$,
but more than that, they share a \emph{disjoint extension commitment}.
Let us try to explain this informal notion with a case analysis.

Suppose we want to construct a computable $\rt^1_2$-instance~$f : \Nb \to 2$ which diagonalizes
against two opponents~$W^f_0$ and~$W^f_1$. After some finite amount of time, 
each opponent~$W^f_i$ will have outputted
a finite approximation of a solution to~$f$, that is, a 
finite $f$-homogeneous set~$F_i$.
The two opponents share a common strategy. $W^f_0$ tries to build an infinite $f$-homogeneous set~$H_0$
for color~0, and~$W^f_1$ tries to build an infinite $f$-homogeneous set~$H_1$ for color~1.
It is therefore difficult to defeat both opponents at the same time, since
if from now on we set~$f(x) = 1$, $W^f_1$ will succeed in extending~$F_1$ to an infinite $f$-homogenenous set,
and if we always set $f(x) = 0$, $W^f_0$ will succeed with its dual strategy.

Consider now the same situation, where we want to construct a computable~$\tto^1_2$-instance $f : 2^{<\Nb} \to 2$.
After some time, the opponent~$W^f_0$ will have outputted a finite tree~$S_0 \cong 2^{<b}$ which is
$f$-homogeneous for color~$0$, and the opponent~$W^f_1$ will have done the same with a finite tree~$S_1 \cong 2^{<b}$
$f$-homogeneous for color~$1$.
The main difference with the~$\rt^1_2$ case is that each opponent will \emph{commit to extend}
each leaf of his finite tree~$S_i$ into an infinite tree isomorphic to~$2^{<\Nb}$.
In particular, for each tree~$S_i$, the sets~$X_\sigma$ of nodes extending the leaf~$\sigma \in S_i$
are \emph{pairwise disjoint}. Therefore, each opponent commits to extend its partial solution
to disjoint sets. Moreover, by asking~$b$ to be large enough, each opponent will commit
to extend enough pairwise disjoint sets so that we can choose two of them for each opponent
and operate the diagonalization without any conflict.

This combinatorial property works in the same way for~$\ers$-instances.
Indeed, in this case, each opponent will commit to extend its partial solution
to pairwise disjoint intervals due to the density requirement of an~$\ers$-solution.
Since the combinatorial arguments of the Erd\H{o}s-Rado theorem and the tree theorem for pairs
are very similar, one may wonder whether they are equivalent in reverse mathematics.

\begin{question}
How do $\erp$ and~$\tto^2_2$ compare over~$\rca$?
\end{question}

The failure of Seetapun's argument for~$\erp$ comes 
from this disjoint extension commitment feature. 
In particular, it is hard to find a forcing notion
for~$\erp$ whose conditions are extendible.

\begin{question}
Does~$\erp$ imply~$\aca$ over~$\rca$?
\end{question}

$\ers$ and~$\tto^1$ have the same state of the art due to their common
combinatorial flavor. However, when looking at their statements for pairs,
$\erp$ and~$\tto^2_2$ have a fundamental difference: $\erp$ has only
a half disjoint extension commitment feature. This weaker property
prevents one from separating~$\rt^2_2$ from~$\erp$ over~$\rca$
by adapting the argument of~$\tto^2_2$ in~\cite{Patey2015strength}.

\begin{question}
Does~$\rt^2_2$ imply $\erp$ over~$\rca$?
\end{question}

We have seen in section~\ref{sect:pigeonhole} that the separation of~$\bst$ from~$\ers$
is directly adaptable from the separation of~$\bst$ from~$\tto^1$
from Corduan, Groszek, and Mileti \cite{Corduan2010Reverse}, since
the combinatorial core of this separation comes from this
shared disjoint extension commitment. It is natural to conjecture
that the status of~$\ers$ with respect to~$\ist$
will be the same as $\tto^1$.

\begin{question}
Does~$\ers$ imply~$\ist$ over~$\rca$?
\end{question}

It is worth mentioning that $\rca+\ist$ proves a
strengthening of both $\tto^1$ and $\ers$, namely the statement ``For every $n$ and 
every $f\colon2^{<\N}\to n$ there exists a strong copy $S$ of the full binary tree such 
that $f$ is constant on $S$'', where by \emph{strong copy} we mean an isomorphic copy of 
$2^{<\N}$ with respect to order and minima. It is easy to see that a strong 
copy computes a dense set of $2^{<\Nb}$, when $2^{<\N}$ is equipped with the standard 
dense linear ordering on binary strings, i.e., the only linear order such that
$\{\tau\colon \tau \succeq 
\sigma\conc0\}<_\Qb\sigma<_\Qb\{\tau\colon\tau\succeq\sigma\conc1\}$ for all 
$\sigma\in2^{<\Nb}$. It is likely that if we can separate  $\tto^1$ or $\ers$ from 
$\ist$, then we can already separate this stronger statement by essentially the same 
proof.

\bibliographystyle{plain}
\bibliography{bibliography}

\end{document}